\documentclass{amsart}
\usepackage[utf8]{inputenc}
\usepackage{amssymb}
\usepackage{hyperref}
\usepackage[all]{xy}
\usepackage[final]{showkeys} 

\theoremstyle{definition}
\newtheorem{mydef}{Definition}
\newtheorem{lem}[mydef]{Lemma}
\newtheorem{thm}[mydef]{Theorem}

\newtheorem{cor}[mydef]{Corollary}

\newtheorem{prop}[mydef]{Proposition}
\newtheorem{defin}[mydef]{Definition}
\newtheorem{example}[mydef]{Example}
\newtheorem{remark}[mydef]{Remark}

\newtheorem{fact}[mydef]{Fact}

\newcommand{\Hilb}{\operatorname{\mathbf{Hilb}}}
\newcommand{\Hilbr}{\Hilb_{m}}
\newcommand{\Set}{\operatorname{\bf Set}}
\newcommand{\Met}{\operatorname{\bf Met}}
\newcommand{\Ban}{\operatorname{\bf Ban}}
\newcommand{\Ab}{\operatorname{\bf Ab}}

\newcommand{\khdual}{\operatorname{\mathbf{KH}}^{op}}
\newcommand{\khorddual}{\operatorname{\mathbf{KH}}_\leq^{op}}
\newcommand{\CAlg}{\operatorname{\textbf{C}\textbf{C}^{\ast}\textbf{Alg}}}
\newcommand{\Mod}{\operatorname{\bf Mod}}
\newcommand{\Ll}{\mathbb{L}}
\newcommand{\ck}{\mathcal{K}}
\newcommand{\cl}{\mathcal{L}}

\author[Lieberman]{Michael Lieberman}
\email{qmlieberman@vutbr.cz}
\urladdr{https://math.fme.vutbr.cz/Home/lieberman}
\address{Institute of Mathematics, Faculty of Mechanical Engineering, Brno University of Technology, Brno, Czech Republic}

\author[Rosick\'y]{Ji\v r\'i Rosick\'y}
\email{rosicky@math.muni.cz}
\urladdr{http://www.math.muni.cz/\textasciitilde rosicky/}
\address{Department of Mathematics and Statistics, Faculty of Science, Masaryk University, Brno, Czech Republic}
\thanks{The second author is supported by the Grant agency of the Czech republic under the grant 19-00902S}

\author[Vasey]{Sebastien Vasey}

\title{Hilbert spaces and $C^\ast$-algebras are not finitely concrete}
\date{\today }

\begin{document}
\begin{abstract}
  We show that no faithful functor from the category of Hilbert spaces with injective linear contractions into the category of sets preserves directed colimits. Thus Hilbert spaces cannot form an abstract elementary class, even up to change of language. We deduce an analogous result for the category of commutative unital $C^\ast$-algebras with $\ast$-homomorphisms. This implies, in particular, that this category is not axiomatizable by a first-order theory, a strengthening of a conjecture of Bankston.
\end{abstract}

\maketitle

\section{Introduction}

In what follows, we prove that a number of categories arising in functional analysis---in particular, Hilbert spaces, Banach spaces, commutative $C^\ast$-algebras, and the dual of the category of compact Hausdorff spaces---are essentially nonaxiomatizable using the tools of discrete (as opposed to \emph{continuous}) model theory, be it classical, infinitary, or abstract.  To be precise, we show that nonaxiomatizability is intrinsic to these categories, in the sense that it holds even up to arbitrary changes in the signature used (or, equivalently, in the choice of underlying sets).  This is accomplished by an entirely novel approach, pioneered in \cite{henry-aec-uncountable-v1} and fully developed here, in which this question is analysed via the failure of \emph{finite concreteness}; that is, the nonexistence of directed colimit preserving underlying set functors.  As we will see, this reduces proofs of nonaxiomatizability (in the strong sense described above) to straightforward counting arguments.  We stress that this method applies to classes beyond those considered here, and is of independent interest.

We will introduce finite concreteness in detail in Section~\ref{secfconcr}: for now, we remark simply that it is an important measure of the concreteness of an abstract category.  In particular, a category $\ck$ is finitely concrete if there exists a forgetful functor $U:\ck\to\Set$ that preserves directed colimits (otherwise known as \emph{direct limits}).  That is, there is an \emph{underlying set functor} $U$ with the property that the underlying set $UX$ of the colimit $X$ of a directed system $\langle X_i\,|\,i\in I\rangle$ in $\ck$ is precisely the colimit of the underlying sets $UX_i$.  In other words, with respect to $U$, colimits in $\ck$ are $\Set$-like.  In algebra, of course, we are accustomed to working with categories of structured sets, in which case there is an obvious choice of $U$, and this $U$ will, more often than not, witness finite concreteness of the category.  In the case of structures arising in functional analysis---take $\Hilbr$, the category of Hilbert spaces and injective linear contractions, for example---there is also an obvious choice of underlying set functor $U$, but this functor will not preserve directed colimits: in general, the underlying set of a directed colimit of complete metric spaces will be the \emph{completion} of the union of the corresponding underlying sets.  

It is natural to ask whether this is a problem intrinsic to the category $\Hilbr$, or whether the situation can be resolved by a clever choice of an alternative functor $U$.\footnote{We note that a change of $U$ amounts to a change in the signature, see e.g. \cite[3.5]{ct-accessible-jsl}.} That is, we consider whether $\Hilbr$ is \emph{finitely concrete}, in the sense that there exists \emph{some} functor $U$ that, like the usual forgetful functor, is faithful---for any linear isometries $f,g$, $Uf=Ug$ only if $f=g$---and succeeds where the usual forgetful functor fails; that is, it preserves directed colimits from $\Hilbr$ into $\Set$.  

Although obvious in retrospect, the connection of this idea to axiomatizability was been made only recently, in \cite{henry-aec-uncountable-v1}. The key observation is that categories axiomatizable in tractable (discrete) logics---finitary first-order logic or, indeed, infinitary logics of the form $\Ll_{\kappa,\omega}$, where we permit conjunctions and disjunctions of fewer than $\kappa$ formulas---are all finitely concrete.  Moreover, \emph{abstract elementary classes} (or \emph{AECs}), which are the central focus of abstract (but still discrete) model theory, are finitely concrete as well.  This gives a clear method of testing for axiomatizability in the above senses: if an abstract category can be shown not to be finitely concrete, it is not susceptible to any of these logical treatments.

As it happens, we prove here that $\Hilbr$ is not finitely concrete, via a very simple counting argument.  As suggested above, this implies that not only is $\Hilbr$ not an AEC with respect to the usual underlying set functor---which is clear from the failure of the union of chains axiom---this problem is essential: $\Hilbr$ is \emph{not} equivalent to an AEC, or a category axiomatizable in $\Ll_{\kappa,\omega}$ for any $\kappa$, and thus, clearly, not axiomatizable in finitary first-order logic.  This answers the open question of \cite[5.10]{multipres-pams}. Moreover, by embedding $\Hilbr$ into various categories, we can deduce more examples of non-finite concreteness (and thus nonaxiomatizability) with no additional effort. In particular, we show that the category $\CAlg$ of commutative unital $C^\ast$-algebras with (unit-preserving) $\ast$-homomorphisms is not finitely concrete (Theorem \ref{cast-thm}), hence, in particular, not elementary. This answers a question of \cite{rosicky-elementary} and, more recently, \cite[1.5]{marra-reggio}.

The same is true, naturally, of the equivalent category $\khdual$, the dual of the category of compact Hausdorff spaces $\khdual$. The latter result significantly strengthens the a longstanding conjecture of Bankston (\cite{bankston, bankston-up}) that $\khdual$ is not $P$-elementary; that is, it not equivalent to the closure under products of the category of models of a first order theory (this original conjecture was solved independently by Banaschewski \cite{banaschewski} and the second author \cite{rosicky-elementary}).  We note in passing that our nonaxiomatizability result also passes to the category of compact ordered spaces, $\khorddual$, considered in, e.g. \cite{abreg}.

These results should be measured against the various (partial) positive axiomatizability results in the literature: \cite{abbad} and \cite{abreg} realize $\khorddual$ as $\aleph_1$-ary varieties, and \cite{boney-pres-metric-mlq, ackerman-pres} each give near-equivalences between continuous classes such as $\Hilbr$ and AECs.  They may also be taken as a testament to the usefulness---indeed, the necessity---of \emph{continuous logic} and related approaches in the analysis of these classes (see, for example, \cite{changk} and the more recent \cite{byu} and \cite{bebyhu}).  Such approaches lack the difficulties outlined here, and allow a great deal of model-theoretic machinery to be brought to bear: see, for example, the sequence of papers \cite{fone}, \cite{ftwo}, and \cite{fthree} on the model theory of operator algebras.

The authors are extremely grateful to James Hanson and the anonymous referee for a very valuable exchange of ideas in connection with the proof of Lemma \ref{inter-lem}.

\section{Finite concreteness}\label{secfconcr}

We assume a basic familiarity with category theory (as presented, for example, in \cite{joy-of-cats}).  Familiarity with accessible categories (see \cite{adamek-rosicky}) and their applications in model theory (\cite{bfo-accessible-v1}, for example) would be useful, but is not essential.

We recall in passing that a \emph{concrete category} consists of a category $\ck$ equipped with a faithful functor $U:\ck\to\Set$.  We refine this notion in the following way:

\begin{defin}\label{deffconcr}
	Let $\ck$ be a category.  We say that $\ck$ is \emph{finitely concrete} if there is a faithful functor $U:\ck\to\Set$ that preserves directed colimits.
\end{defin}

As a simple example, we note:
\begin{example}
	The category of abelian groups and injective homormorphisms, $\Ab_m$, is finitely concrete, a property which is witnessed by the standard forgetful functor $U:\Ab_m\to\Set$.  Given a directed system $\{G_i\,|\,i\in I\}$---we suppress the mappings, for simplicity---with colimit $G$ in $\Ab_m$, $U(G)$ is simply the directed union of the $U(G_i)$, which is precisely the directed colimit in $\Set$.  So $U$ does indeed preserve directed colimits.  The same holds for many familiar algebraic categories.
\end{example} 

As our template for categories susceptible to analysis through discrete logic, we take a very general notion:

\begin{defin}\label{defaecat}
	Let $\ck$ be an abstract category.  Following \cite[5.3]{beke-rosicky}, we say that $\ck$ is an abstract elementary category if it has the following properties:
	\begin{enumerate}
		\item $\ck$ is accessible and has all directed colimits, and
		\item $\ck$ admits an embedding $F:\ck\to\cl$, $\cl$ finitely presentable, such that
		\begin{enumerate}
		\item $F$ is iso-full: any isomorphism $f:FA\to FB$ in $\cl$ is $F(\bar{f})$ for some $\bar{f}:A\to B$ in $\ck$.
		\item $F$ is coherent: For any commutative triangle 
		$$\xymatrix{FA\ar[rr]^{H(h)}\ar[dr]_f & & FC\\
		 & FB\ar[ur]_{H(g)} & }$$
		 there is $\bar{f}:A\to B$ in $\ck$ with $F(\bar{f})=f$.
		\item $F$ preserves directed colimits.
		\end{enumerate}
	\end{enumerate}
\end{defin}

In a moment, we will connect this to more familiar model-like categories.  For now, we note:

\begin{lem}\label{lemaecatfc}
	Any abstract elementary category is finitely concrete.
\end{lem}

\begin{proof}
	Let $\ck$ be an abstract elementary category.  We know that $\ck$ admits a functor $F:\ck\to\cl$ that preserves directed colimits and, as an embedding of categories, is faithful.  Being finitely accessible, $\cl$ admits a faithful, directed colimit-preserving functor $G:\cl\to\Set$.  The composition $GF$ witnesses finite concreteness of $\ck$.
\end{proof}

As their name suggests, abstract elementary categories are closely related to the abstract elementary classes (AECs) of Shelah, and thus generalize a number of familiar categories arising from logical considerations.

\begin{remark}\label{rmkmodcats}
\begin{enumerate}
		\item\label{defmodcatsaec} Upon parsing the definitions, it is clear that any AEC forms an abstract elementary category.  Indeed, AECs are precisely the abstract elementary categories in which all morphisms are monomorphisms (see \cite[5.7]{beke-rosicky}).
		\item\label{defmodcatsinf} Similarly, following \cite{makkai-pare}, we say that a category $\ck$ is \emph{$\Ll_{\infty,\omega}$-elementary} if it is equivalent to the category of models of a fragment of the infinitary logic $\Ll_{\kappa,\omega}$, for some infinite cardinal $\kappa$.
		\item\label{defmodcatselem} We say that a category $\ck$ is \emph{elementary} if it is equivalent to one of the form $\Mod(T)$, whose objects are models of a finitary first-order theory $T$, and whose morphisms are all $\Ll(T)$-structure homomorphisms.  
\end{enumerate}
\end{remark}

All of the above categories are finitely concrete, by Lemma~\ref{lemaecatfc}.  This was already manifestly the case, of course: they each come equipped with a natural underlying set functor that witnesses their finite concreteness.

\begin{cor}\label{corfconcrax}
	If a category $\ck$ is not finitely concrete, it is not equivalent to an abstract elmentary category (nor to an AEC, an $\Ll_{\infty,\omega}$-elementary category, or an elementary category).
\end{cor}

\begin{prop}\label{propfconcrtrans} Let $F:\ck^{\ast}\to\ck$ be a faithful functor that preserves directed colimits.  If $\ck$ is finitely concrete, so is $\ck^{\ast}$.\end{prop}
\begin{proof}
	Let $U:\ck\to\Set$ be a faithful, directed-colimit-preserving functor, the existence of which follows from finite concreteness of $\ck$.  The composition $FU:\ck^{\ast}\to\Set$ will also be faithful and preserve directed colimits, witnessing finite concreteness of $\ck^{\ast}$.
\end{proof}

Of greater use to us will be the contrapositive of Proposition~\ref{propfconcrtrans}: if a category $\ck^{\ast}$ admits a faithful, directed-colimit-preserving functor into a \emph{non-finitely concrete} category $\ck$, then $\ck^{\ast}$ is not finitely concrete, either.  This will play an essential role in Section~\ref{secrest}, where we take advantage of the interembeddability of $\Hilbr$, $\Ban$ (the category of Banach spaces), $\CAlg$, and $\khdual$.

\section{Hilbert spaces}\label{sechilb}

Henceforth, we define $\Hilb$ to be the category whose objects are complex Hilbert spaces and whose morphisms are linear contractions. Note that injective linear contractions are exactly the \emph{monomorphisms} in $\Hilb$. Thus we let $\Hilbr$ denote the subcategory of $\Hilb$ with the same objects, but with injective linear contractions as morphisms. We observe in passing that $\Hilbr$ is an $\aleph_1$-accessible category with directed colimits.

A linear mapping $A\to B$ of Hilbert spaces is a contraction if and only if
$(a,a)\geq (fa,fa)$ for every $a\in A$, where $(-,-)$ denotes the inner product. Linear isometries (equivalently orthogonal operators) are linear mappings preserving the inner product.
 
Our goal is to show that $\Hilbr$ is not finitely concrete; that is, there is no faithful, directed-colimit-preserving functor $U:\Hilbr\to\Set$.  Suppose, to the contrary, that there \emph{is} such a functor $U$.

The following definition follows \cite[4.3]{henry-aec-uncountable-v1}.

\begin{defin}
   Let $A$ be a Hilbert space, and let $x \in U A$. We say that $x$ is \emph{supported} on a subspace $A_0$ of $A$ if whenever $f, g: A \to B$
are injective linear contractions such that $f i_{A_0, A} = g i_{A_0, A}$ (where $i_{A_0, A}$ denotes the inclusion map $A_0\to A$), then $f (x) = g (x)$. When $A$ is clear from context, we omit it.
\end{defin}

Note that, as is standard, we have abused notation and written $f (x)$ instead of $(U f)(x)$. The next observation will be used repeatedly:

\begin{remark}\label{supp-rmk}
  If $A_0$ is a subspace of a Hilbert space $A$ and $x_0 \in U A_0$, then $i_{A_0, A} (x_0)$ is supported on $A_0$. 
\end{remark}

\begin{lem}\label{fd-lem}
  Let $A$ be a Hilbert space, and let $x \in U A$. Then $x$ is supported on some finite dimensional subspace of $A$.
\end{lem}
\begin{proof}
  $A$ is a directed colimit of its finite-dimensional subspaces. Since $U$ preserves directed colimits, $U A$ is a directed colimits of sets of the form $U A_0$, for $A_0$ a finite-dimensional subspace of $A$. So $x = i_{A_0, A} (x_0)$ for some finite-dimensional $A_0$ and some $x_0 \in U A_0$. Now use Remark \ref{supp-rmk}.
\end{proof}

We now prove that finite-dimensional supports are closed under intersections. This is crucial and not so obvious, since there is no assumption that the concrete functor preserves pullbacks (which are intersections here). The proof follows \cite[4.7]{henry-aec-uncountable-v1}. We will need the following:  

\begin{lem}\label{supp-cond}
Let $A$ be a Hilbert space and $0<\delta\leq 1$. Then $x\in UA$ is supported on $A_0$ if and only if whenever $f,g:A\to B$ are injective
linear maps of norm $\leq\delta$ such that $f i_{A_0, A} = g i_{A_0, A}$, then $f (x) = g (x)$. 
\end{lem}
\begin{proof}
Assume the condition from the statement and consider injective linear contractions $f,g:A\to B$ such that $f i_{A_0, A} = g i_{A_0, A}$. Then
$\parallel \delta f\parallel,\parallel\delta g\parallel\leq\delta$ and
$\delta f i_{A_0, A} = \delta g i_{A_0, A}$. Thus 
$\delta f(x)=\delta g(x)$, hence $f(x)=g(x)$.
\end{proof}

\begin{lem}[\cite{GK} Lemma 2.1]\label{contr}
Let $e_1,\dots,e_n$ be a basis of a finite-dimensional Banach space $A$.
There is $\delta>0$ such that for every Banach space $B$ and every linear map $f:A\to B$ the following implication holds:
$$
\max_{i=1,\dots,n}\parallel f(e_i)\parallel\leq\delta\Rightarrow \parallel f\parallel\leq 1.
$$
\end{lem}

\begin{remark}\label{sum}
Let $C$ be a closed subspace of a Hilbert space $A$ and $C^\perp$ its orthogonal complement. Then $A=C\oplus C^\perp$ (see \cite[1.4.6]{BS}).
Let $f_1:C\to B$ and $f_2:C^\perp\to B$ be injective linear contractions 
such that $(f_1x,f_2y)=0$ for every $x\in C$ and $y\in C^\perp$. 
Let $f:A\to B$ be the induced linear map, i.e., $f(x+y)=f_1(x)+f_2(y)$. Then $f$ is an injective linear contraction.

Indeed,
$$
(x+y,x+y)=(x,x)+(y,y)\geq (f_1x,f_1x)+(f_2y,f_2y)=
(f(x+y),f(x+y)).
$$
Hence $f$ is a linear contraction. Assume that $f(x+y)=0$. Then
$f_1(x)=-f_2(y)$, hence
$$
(f_1x,f_1x)=(f_1x,-f_2y)=0.
$$
Therefore, $f_1(x)=0$, i.e., $x=0$. Then $f_2(y)=0$, hence $y=0$. Consequently, $x+y=0$.
\end{remark}

\begin{lem}\label{inter-lem}
Let $A$ be an infinite-dimensional Hilbert space, let $x \in U A$ and let $A_0, A_1$ be finite-dimensional subspaces of $A$. If $x$ is supported on both $A_0$ and $A_1$, then $x$ is supported on $A_0 \cap A_1$.
\end{lem}
\begin{proof}
 Fix a Hilbert space $B$. We can assume that $\dim A\leq\dim B$ because, otherwise, there is no injective linear map $A\to B$. For injective linear contractions $f, g: A \to B$, we write $f \sim^\ast g$ if either 
$f i_{A_0, A} = g i_{A_0, A}$ or $f i_{A_1, A} = g i_{A_1, A}$. This is usually not an equivalence relation, so let $\sim$ be its transitive closure. Observe that since $x$ is supported on both $A_0$ and $A_1$, we have that $f \sim^\ast g$ implies $f (x) = g (x)$, hence also $f \sim g$ implies $f (x) = g (x)$.

Fix orthonormal bases $\mathcal{B}_0$ and $\mathcal{B}_1$ for $A_0$ and $A_1$, respectively, so that $\mathcal{B}_0 \cap \mathcal{B}_1$ is a basis for $A_0 \cap A_1$. Let $A_2=\langle A_0\cup A_1\rangle$. Since $A_2$ is a closed subspace of $A$ (\cite[1.2.7]{BS}, $A=A_2\oplus A_2^\perp$. Let $\mathcal{C}$ be an orthonormal base of $A_2^\perp$ in $A$.
Take $0<\delta\leq 1$ from Lemma \ref{contr} for $A_2$. Let $f, g: A \to B$ be injective linear maps of norm $\leq\delta$ given so that $f i_{A_0 \cap A_1, A} = g i_{A_0 \cap A_1, A}$. We will find an injective linear contraction $h: A \to B$ so that $f \sim h$ and $g \sim h$, which will imply that $f \sim g$, and hence that $f(x) = g(x)$, by the observation above. Then, following Lemma \ref{supp-cond}, $x$ is supported on $A_0\cap A_1$.

Let $S$ be the orthogonal complement of the space 
$\langle f[A_0] \cup g[A_1]\rangle$ generated by $f[A_0] \cup g[A_1]$
in $B$. Since $\dim A\leq\dim S$, there is a linear isometry 
$s:A\to S$. Hence $\parallel\delta s\parallel\leq\delta$. Let $h:A\to B$ be a linear map which restricts to $f$ (hence also $g$) on 
$\mathcal{B}_0\cap\mathcal B_1$ and restricts to $\delta s$ on $(\mathcal{B}_0\cup\mathcal{B}_1)\setminus(\mathcal{B}_0\cap\mathcal{B}_1)\cup\mathcal{C}$. To see that $h$ is an injective contraction, we apply Remark \ref{sum}: the restriction of $h$ on  $A_0\cap A_1$ is $f$, the restriction of $h$ on $(A_0\cap A_1)^\perp$ is $\delta s$, and 
$(hx,hy)=0$ for $x\in A_0\cap A_1$ and  $y\in(A_0\cap A_1)^\perp$.

We claim that $f \sim h \sim g$. We prove that $f \sim h$, and a symmetric argument will prove $g \sim h$.  Let $f':A\to B$ be a linear map which is $f$ on $\mathcal{B}_0$ and $h$ on $(\mathcal{B}_1\setminus\mathcal{B}_0)\cup\mathcal {C}$. Following Lemma \ref{contr}, the restriction of $f'$ on $A_2$ is a contraction. Assume that
$f'(x)=0$ for $x\in A_2$. Since $x=y+z$ where $y\in A_0$ and 
$z\in\langle \mathcal{B}_1\setminus\mathcal{B}_0\rangle$, we have
$f'(y)=-f'(z)$. Since $(f'y,f'z)=0$, we have $f'(y)=f'(z)=0$, hence $y=z=0$, i.e., $x=0$. Thus the restriction of $f'$ on $A_2$ is injective.
Since the restriction of $f'$ on $A_2^\perp$ coincides with that of $\delta s$, it is an injective contraction too. Since $(f'x,f'y)=0$ for
$x\in A_0$ and $y\in A_0^\perp$, Remark \ref{sum} implies that $f'$ is
an injective contraction. Observe that $f\sim f'\sim  h$ by definition, hence $f \sim h$, as desired.
\end{proof}

\begin{lem}\label{fd-lem-uq}
  For every Hilbert space $A$ and any $x \in U A$, there is a unique minimal finite-dimensional subspace $A_0$ of $A$ on which $x$ is supported.
\end{lem}
\begin{proof}
  Combine Lemmas \ref{fd-lem} and \ref{inter-lem} with the fact that a nontrivial intersection of two finite-dimensional subspaces must have lower dimension.
\end{proof}

\begin{defin}
  For any Hilbert space $A$ and $x \in U A$, we call the minimal subspace of $A$ on which $x$ is supported (given by Lemma \ref{fd-lem-uq}) the \emph{support of $x$ (in $A$)}. We say that this support is \emph{trivial} if it is the zero space, and \emph{nontrivial} otherwise.
\end{defin}

\begin{lem}\label{trivial-supp}
  For any nonzero subspace $A_0$ of an infinite-dimensional Hilbert space $A$, there is $x_0 \in U A_0$ such that $i_{A_0,A} (x_0)$ has nontrivial support in $A$.
\end{lem}
\begin{proof}
  Suppose not; that is, suppose that there is a nonzero subspace $A_0$ of such an $A$ with the property that every $x_0\in U A_0$ has trivial support in $A$. Let $f,g: A \to B$ be any two morphisms such that $\dim A<\dim B$ and $f i_{A_0, A} \neq g i_{A_0, A}$.  In particular, we may take $B$ to be the direct sum of $A$ with a Hilbert space of higher dimension extending $A$. Let $f$ send $A$ to its copy in the left component of the direct sum and $g$ send $A$ to its copy in the right component.  We know that $f$ and $g$ agree on the zero space, on which every $x_0 \in U A_0$ is supported in $A$, under our current assumptions.  It follows that for all $x_0\in U A_0$, $f (i_{A_0, A} (x_0)) = g (i_{A_0, A} (x_0))$. Thus $U (f i_{A_0, A}) = U (g i_{A_0, A})$ and so $f i_{A_0, A} = g i_{A_0, A}$, by faithfulness of $U$.  This contradicts the construction of $f$ and $g$.
\end{proof}

\begin{thm}
  No faithful functor from $\Hilbr$ to $\Set$ preserves directed colimits. 
\end{thm}
\begin{proof}
  Suppose for a contradiction that $U$ is a faithful functor from $\Hilbr$ to $\Set$ preserving directed colimits. By the uniformization theorem \cite[2.19]{adamek-rosicky} (and see \cite[4.3]{beke-rosicky}), there is a cardinal $\mu_0$ such that for all regular cardinals $\mu \ge \mu_0$, $U$ preserves $\mu$-presentable objects. Fix a cardinal $\lambda > \mu_0 + 2^{\aleph_0}$ of countable cofinality. Let $A$ be the Hilbert space of dimension $\lambda$, hence of cardinality $\lambda^{\aleph_0} > \lambda$. Note that $A$ is $\lambda^+$-presentable, so by definition of $\mu_0$ we also have that $U A$ is $\lambda^+$-presentable, hence has cardinality at most $\lambda$.

  Each nonzero element of $A$ spans a line (i.e.\ a one-dimensional subspace of $A$), and each line contains only $|\mathbb{C}| = 2^{\aleph_0}$-many elements. This implies that there are $\lambda^{\aleph_0}$-many distinct lines. Since $|U A| \le \lambda < \lambda^{\aleph_0}$, there must be a line $A_0$ that is \emph{not} the support of any $x \in U A$. However, for each $x_0 \in U A_0$, $i_{A_0, A} (x_0)$ is supported on $A_0$ (Remark \ref{supp-rmk}). By minimality of the support, the support of every element of $i_{A_0, A}[U A_0]$ must be a strict subspace of $A_0$, i.e.\ the zero space. In other words, every element of $i_{A_0, A}[U A_0]$ has trivial support. This contradicts Lemma \ref{trivial-supp}.  
\end{proof}

\begin{cor}\label{corhilbnoaecat} In light of Corollary~\ref{corfconcrax}, $\Hilbr$ is not an abstract elementary category.\end{cor}

To emphasize: Corollary~\ref{corhilbnoaecat} implies that the category of Hilbert spaces cannot be realized as an AEC, even up to arbitrary changes of signature.

\begin{remark}
Linear isometries are exactly the \emph{regular monomorphisms} in $\Hilb$ (see \cite[7.58(3)]{joy-of-cats}). Thus we let $\Hilb_r$ denote the subcategory of $\Hilb$ with the same objects, but with linear isometries as morphisms. We observe in passing that $\Hilb_r$ is an $\aleph_1$-accessible category with directed colimits. The proof of Lemma
\ref{inter-lem} fails for isometries.  We believe that $\Hilb_r$
is not finitely concrete but we do not currently have a proof.
\end{remark}

\section{$C^\ast$-algebras and other examples}\label{secrest}

We will now make use of the non-finite-concreteness of $\Hilbr$ to obtain nonaxiomatizability results for a number of other categories.  In each case, we make use of the well-behaved embeddings of $\Hilbr$ into the relevant category, together with Proposition~\ref{propfconcrtrans}.

\begin{example}\label{banach} \
\begin{enumerate}
  \item\label{banach-1} The category $\Met_m$ of complete metric spaces with injective contractions and the category $\Ban_m$ of Banach spaces with injective linear contractions, each admit faithful, directed-colimit-preserving embeddings of $\Hilbr$, and are therefore not finitely concrete.  They are thus nonaxiomatizable in precisely the same sense as $\Hilbr$.

  \item\label{banach-2}  The category $\Hilb$ of Hilbert spaces with linear contractions is not finitely concrete. Indeed, the inclusion $\Hilbr \to \Hilb$ is faithful and preserves directed colimits. This applies more generally, in fact: the same is true any time we have a non-finitely-concrete subcategory $\ck$ of a category $\ck^\ast$ that is closed under directed colimits (that is, where the inclusion preserves directed colimits). In particular, we also get that the category $\Met$ of complete metric spaces with contractions and the category $\Ban$ of Banach spaces with linear contractions are not finitely concrete, thus cannot be abstract elementary categories.
\end{enumerate}
\end{example}

Let $\CAlg$ be the category of commutative unital $C^\ast$-algebras and unit-preserving $\ast$-homomorphisms. 

\begin{thm}\label{cast-thm}
The category $\CAlg$ is not an abstract elementary category.
\end{thm}
\begin{proof}
Let $V:\CAlg\to\Ban$ be the forgetful functor (recall that $\ast$-homomorphisms are, in particular, contractions). It is folklore (see for example \cite[\S12]{pestuniv}) that $V$ has a left adjoint $F:\Ban\to\CAlg$. We note that this also follows from the adjoint functor theorem for locally presentable categories (\cite[1.66]{adamek-rosicky}) because $V$ preserves limits and $\aleph_1$-directed colimits and both $\Ban$ and $\CAlg$ are locally presentable (see, respectively, \cite[1.48]{adamek-rosicky}, and \cite[3.28]{adamek-rosicky}---in the second case we need the result of Isbell \cite{isbell} that $\CAlg$ is a variety of algebras with $\aleph_0$-ary operations).
  
Moreover, the components of the unit of the adjunction, $\eta_B:B\to VFB$, are linear isometries hence, in particular, monomorphisms. This follows from the fact that any Banach space $B$ can be isometrically embedded into a commutative unital $C^\ast$-algebra. Indeed, this algebra can be taken to be the $C^\ast$-algebra $C(X)$ of continuous complex-valued functions on the closed unit ball $X$ of the dual space $B^\ast$ with the weak* topology. Since $X$ is compact (by the Banach-Alaoglu theorem), $C(X)$ is commutative and unital. 

The functor $F$ is faithful by a general result for adjoints (see \cite[19.14(1)]{joy-of-cats}). Moreover, since $F$ is a left adjoint, it preserves arbitrary colimits.  By Example \ref{banach}(\ref{banach-2}), then, $\CAlg$ is not finitely concrete.
\end{proof}

\begin{fact}
	The category $\CAlg$ is not an abstract elementary category.
\end{fact}

As $\khdual$, the dual of the category of compact Hausdorff spaces and continuous maps, is equivalent to $\CAlg$, we have, immediately:

\begin{cor}\label{khcor}
	The category $\khdual$ is not an abstract elementary category.
\end{cor}

As noted in the introduction, this strengthens existing results of \cite{banaschewski} and \cite{rosicky-elementary}---that $\khdual$ is not $P$-elementary.  As noted in \cite[23]{abreg}, it follows immediately that the analogous result holds for $\khorddual$, the dual of the category of compact \emph{ordered} spaces, as well. 

\bibliographystyle{amsalpha}
\bibliography{hilb-mono}
\end{document}